\documentclass[11pt, oneside]{amsart}   	% use "amsart" instead of "article" for AMSLaTeX format
\usepackage{geometry}                		% See geometry.pdf to learn the layout options. There are lots.
\usepackage{amssymb}
\usepackage{hyperref}
\usepackage{geometry}                		% See geometry.pdf to learn the layout options. There are lots.
\usepackage{amssymb,amsthm,amsmath}
\usepackage{hyperref}
\usepackage{xypic}
\usepackage{tikz}
\usetikzlibrary{calc}

\address{{\bf  Yoav Len}\newline Deaprtment of Mathematics,  University of T\"ubingen}
\email{yoavlen@gmail.com}

\newtheorem{theorem}{Theorem}
\numberwithin{theorem}{section}

\newtheorem{proposition}[theorem]{Proposition}

\newtheorem{lemma}[theorem]{Lemma}

\newtheorem*{proposition*}{Proposition}
\newtheorem*{theorem*}{Theorem}
\newtheorem*{lemma*}{Lemma}
\newtheorem*{claim*}{Claim}
\newtheorem*{conjecture*}{Conjecture}

\theoremstyle{definition}
\newtheorem{definition}[theorem]{Definition}

\theoremstyle{remark}
\newtheorem{example}[theorem]{Example}

\newtheorem*{example*}{Example}

\newcommand{\calC}{\mathfrak{C}}
\newcommand{\calD}{\mathcal{D}}
\newcommand{\calK}{\mathcal{K}}
\newcommand{\calE}{\mathcal{E}}
\newcommand{\calP}{\mathcal{P}}
\newcommand{\calQ}{\mathcal{Q}}
\newcommand{\calR}{\mathcal{R}}

\newcommand{\calf}{\mathfrak{f}}
\newcommand{\calg}{\mathfrak{g}}

\title{Hyperelliptic graphs and metrized complexes}
\author{Yoav Len}
\begin{document}
\maketitle
\begin{abstract}
We prove a version of Clifford's theorem for metrized complexes. Namely, a metrized complex that carries a divisor of degree $2r$ and rank $r$ (for $0<r<g-1$) also carries a divisor of degree $2$ and rank $1$. We provide a structure theorem for hyperelliptic metrized complexes, and use it to classify divisors of degree bounded by the genus. We discuss a tropical version of Martens' theorem for metric graphs.
\end{abstract}

\section{Introduction}
For a smooth algebraic curve, Clifford's theorem states that a divisor of rank $r$ has degree at least $2r$, and when $0<r<g-1$, equality may only be obtained for hyperelliptic curves \cite[Chapter III]{ACGH}. The first part of the theorem follows immediately from Riemann--Roch, whereas the second part requires more subtle geometric methods. 

With the development of tropical and non-archimedean geometry in recent years, it was observed that many theorems from classical algebraic geometry have combinatorial analogs in tropical geometry. Baker and Norine introduced  divisors on finite graphs, and showed that they satisfy a Riemann--Roch theorem \cite{BN}. Their results were later generalized by various authors to metric and weighted graphs \cite{AC,GK, MZ}.  Via Baker's Specialization Lemma \cite{Baker}, such results provide new combinatorial techniques for studying algebraic curves. 

Subsequently, Amini and Baker introduced metrized complexes, a common generalization of metric graphs and algebraic curves \cite{AB}. While these objects tend to be more involved than graphs, they also capture more algebraic information, and provide a much stronger tool in some cases.  
For instance, Katz, Rabinoff and  Zurich-Brown apply the first part of Clifford's theorem for the metrized complex associated to a semistable model of a curve, to bound the number of its rational points, thus proving a weak version of the Bombieri--Lang conjecture \cite{KRZ}. 

Similarly to the algebraic case, an analogue of the first part of Clifford's theorem follows immediately from Riemann--Roch. However, the methods used to prove the second part do not carry well into the tropical world. Nevertheless, it is shown in \cite{Coppens} that the full extent of Clifford's theorem holds for metric graphs. Our main result is an extension of the theorem  to metrized complexes as well.

\begin{theorem*}[\ref{mainTheorem}]
Let $\calC$ be a metrized complex of genus $g$,
% such that the underlying graph $G$ has positive genus $h$, 
and suppose that for some $0<r<g-1$ there is a divisor class $\delta$ of degree $2r$ and rank $r$. Then $\calC$ is hyperelliptic.
\end{theorem*}

By \emph{hyperelliptic} we mean a metrized complex having a divisor of degree $2$ and rank $1$.  For a smooth curve, the existence of such a divisor induces a double cover of a line. Similarly, for a 2-connected metric graph, such a divisor implies the existence of a harmonic map of degree 2 to a tree \cite[Theorem 1.3]{chan}. The analogous statement for metrized complexes is no longer true, as shown by Example 4.14 of \cite{ABBR2} and Remark 5.13 of  \cite{AB}. Nevertheless, just as hyperelliptic graphs can be pictured  as two isomorphic trees meeting at their leaves (\cite[Theorem 1.3]{chan}), we show that hyperelliptic metrized complexes consist of two isomorphic trees, meeting along hyperelliptic algebraic curves. See Lemma \ref{structureHE} for a precise statement. This description allows us to classify all the  effective divisors  whose degree is bounded by the genus. 
\begin{theorem*}[\ref{uniqueHyper}]
Let $\calC$ be a hyperelliptic metrized complex. Then every divisor class $\delta$ of degree $d$ and rank $r$ with $0\leq r\leq d\leq g$ is of the form $r\cdot\calg^1_2 + p_{2r+1}+\ldots+p_d$. 
\end{theorem*}
\noindent In the appendix, we restrict ourselves to metric graphs, and discuss a possible tropical version of Martens' theorem, which is a refinement of Clifford's theorem. The results presented there are joint work with David Jensen.

Our strategy for proving the main theorem is partly inspired by the techniques used in \cite{Coppens}.  However, our argument is entirely self contained. In particular, it provides an independent proof of Clifford's theorem for graphs by considering metrized complexes in which all the components are rational. 
Before delving into the proof and introducing various notations, we begin with an example to demonstrates the strategy.
\begin{example}
Let $\calC$ be the metrized complex of genus 4, depicted in Figure \ref{example} (see Section \ref{sec:preliminaries} for notations regarding metrized complexes). All its edges have the same length,  the points $p_3,q_3,p_4,q_4$ are in the middle of the edges, and  $p,q$ are of of equal distance from $p_2,q_2$ respectively. Suppose that  $p_1+q_1$ and $p_4+q_4$ are equivalent to the pair of nodes corresponding to the incoming edges at their respective components. The points $p,p_1,p_2,p_3,p_4$ form a rank determining set for $\calC$, so in order to show that the complex is hyperelliptic, it suffices to find a divisor of degree $2$ passing through all of them. 

The divisor $\calD = 4\cdot (x)$ has rank $2$, and therefore, has a representative $\calD'$ containing $p_1+p_2$. It is straightforward to check that this representative is exactly $p_1+p_2+q_1+q_2$. Similarly, it has a representative containing $p_2+p_3$, namely, $p_2+p_3+q_2+q_3$. It follows that 
$$
p_1+p_2+q_1+q_2\simeq p_2+p_3+q_2+q_3,
$$
so by linearity, $p_1+q_1\simeq p_3+q_3$.
By repeating this process for all the different combinations $p_i+p_j$ and $p_j+p_k$, we see that 
$$
p_1+q_1\simeq p_2+q_2 \simeq p+q\simeq p_3+q_3 \simeq p_4+q_4.
$$
\noindent Therefore, $p_1+q_1$ has rank $1$, and $\calC$ is hyperelliptic.

\begin{figure}
\centering

\begin{tikzpicture}[scale=.5]

\coordinate (a) at (-9,0);
\coordinate (b) at (0,0);
\coordinate (c) at (9,0);

%Component on the left
%\node at ($(a)+(0,1.5)$) {$C_1$};
\draw (a) ellipse (3cm and 2cm);
\draw[thick] ($(a)+(-1.1,0.05)$) to [out=20,in=160] ($(a)+(1.1,0.05)$);
\draw[thick] ($(a)+(-1.4,0.15)$) to [out=-20,in=-160] ($(a)+(1.4,0.15)$);

%Middle component
\coordinate (b) at (0,0);
\draw (b) circle (2.3cm);
\draw[ thick] ($(b) + (-2.3,0)$) to [out=340,in=200] ($(b)+(2.3,0)$);
\draw[dashed, thick] ($(b) + (-2.3,0)$) to [out=20,in=160] ($(b)+(2.3,0)$);

%Rightmost component
\coordinate (c) at (9,0);
\draw (c) ellipse (3cm and 2cm);
\draw[thick] ($(c)+(-1.1,0.05)$) to [out=20,in=160] ($(c)+(1.1,0.05)$);
\draw[thick] ($(c)+(-1.4,0.15)$) to [out=-20,in=-160] ($(c)+(1.4,0.15)$);

%Edges
\draw[fill] ($(a) +(2,1)$) circle [radius=0.08];
\draw[fill] ($(b) +(-1.5,1)$) circle [radius=0.08];
\draw[ thick] ($(a) +(2,1)$) to [out=20,in=160] ($(b) +(-1.5,1)$);

\draw[fill] ($(a) +(2,-1)$) circle [radius=0.08];
\draw[fill] ($(b) +(-1.5,-1)$) circle [radius=0.08];
\draw[ thick] ($(a) +(2,-1)$) to [out=-20,in=-160] ($(b) +(-1.5,-1)$);

\draw[fill] ($(b) +(1.5,1)$) circle [radius=0.08];
\draw[fill] ($(c) +(-2,1)$) circle [radius=0.08];
\draw[ thick] ($(b) +(1.5,1)$) to [out=20,in=160] ($(c) +(-2,1)$);

\draw[fill] ($(b) +(1.5,-1)$) circle [radius=0.08];
\draw[fill] ($(c) +(-2,-1)$) circle [radius=0.08];
\draw[ thick] ($(b) +(1.5,-1)$) to [out=-20,in=-160] ($(c) +(-2,-1)$);

%Special points
\draw[fill,red] ($(-4.2,1.55)$) circle [radius=0.08];
\node [above] at ($(-4.2,1.55)$) {$p_2$};
\draw[fill,red] ($(-4.2,-1.55)$) circle [radius=0.08];
\node [below] at ($(-4.2,-1.55)$) {$q_2$};

\draw[fill,red] ($(-5.2,1.47)$) circle [radius=0.08];
\node [above left] at ($(-5.2,1.47)$) {$p$};
\draw[fill,red] ($(-5.2,-1.47)$) circle [radius=0.08];
\node [below left] at ($(-5.2,-1.47)$) {$q$};

\draw[fill,red] ($(4.2,1.55)$) circle [radius=0.08];
\node [above] at ($(4.2,1.55)$) {$p_3$};
\draw[fill,red] ($(4.2,-1.55)$) circle [radius=0.08];
\node [below] at ($(4.2,-1.55)$) {$q_3$};

\draw[fill,red] ($(a) +(-1.5,.9)$) circle [radius=0.08];
\node [left] at ($(a)+(-1.4,.9)$) {$p_1$};
\draw[fill,red] ($(a) +(-1.5,-.9)$) circle [radius=0.08];
\node [left] at ($(a)+(-1.4,-.9)$) {$q_1$};

\draw[fill,red] ($(c) +(1.5,.9)$) circle [radius=0.08];
\node [right] at ($(c)+(1.4,.9)$) {$p_4$};
\draw[fill,red] ($(c) +(1.5,-.9)$) circle [radius=0.08];
\node [right] at ($(c)+(1.4,-.9)$) {$q_4$};

\draw[fill,blue] ($(-0.5,1.5)$) circle [radius=0.08];
\node [above left] at ($(-0.4,1.4)$) {$x$};

\end{tikzpicture}

\caption{The metrized complex $\calC$}
\label{example}
\end{figure}
\end{example}
In the next section, we show that every metrized complex of genus $g$ which satisfies the conditions of the main theorem contains a divisor of degree $2g$, with similar properties to the divisor
 $\sum p_i + \sum q_i$ in the example above. 
\\
\\
\noindent\textbf{Acknowledgements.}
I started thinking about this problem some time ago, following a meeting of the Yale lunch seminar, and I thank all the participants for the fruitful discussion. I am grateful to Dave Jensen, Dhruv Ranganathan, Sam Payne and Nick Wawrykow for insightful remarks on a previous version of this paper, and for pointing out a gap in the proof of Proposition \ref{intersection}. In addition, I had enlightening conversations with Spencer Backman, Marc Coppens, Ohad Feldheim and Jifeng Shen.
While writing this paper, I was supported by DFG grant MA 4797/1-2.

\section{Preliminaries} \label{sec:preliminaries}

In what follows, we assume familiarity with the theory of tropical divisors and metrized complexes. We refer the reader to \cite{BJ} for an extensive exposition on the topic, and to \cite{AB} for a more thorough treatment. Roughly speaking, a metrized complex is a generalization of a metric graph,  obtained by placing smooth curves at the vertices, and defining linear equivalence in a way that combines chip firing on the graph and linear equivalence on the curves. 

\begin{definition}
A metrized complex is \emph{hyperelliptic} if it has a divisor of degree $2$ and rank $1$.
\end{definition}

We refer to the algebraic curves placed at the vertices as \emph{components}. For each vertex $v$, denote $C_v$ the corresponding component and $g_v$ its genus. The point of a component $C_v$ associated to an edge is referred to as a \emph{node}. The metric graph $\Gamma$, obtained by removing the components is called the \emph{underlying} graph of $\calC$. There is a natural map which takes divisors on $\calC$ to divisors on $\Gamma$. By abuse of notation, we often identify a divisor on $\calC$ with its image without mention. The elements of a divisor are referred to as \emph{chips}. For an algebraic curve or a metric graph we denote by $W^r_d$ the set of its divisor classes of degree $d$ and rank $r$, or simply $W_d$ when $r=0$.

\begin{definition} 
A divisor on a graph or a metrized complex is said to be \emph{rigid} if it is the unique effective divisor in its class. 
\end{definition} 

On an algebraic curve, a divisor is rigid if and only if it has rank zero. On metric graphs, rank zero is a necessary but not a sufficient condition for rigidness. However, as seen by the following lemma, rigid divisors are ubiquitous.

\begin{lemma}\label{rigidDivisor}
Let $\Gamma$ be a metric graph of genus $h>0$, and let $K$ be its canonical divisor. Then there is a divisor $P$ of degree $h-1$ such that $P$ and $K-P$ are rigid. Moreover, there is an open set of such divisors in the space
 $W_{h-1}(\Gamma)$ of effective divisor classes of degree $h-1$.
\end{lemma}
\begin{proof}
%By \cite[Lemma 3.5]{ABS}, there is an open dense set in $\text{Pic}_h(\Gamma)$ of divisors whose complement is a spanning tree, and in particular are rigid. Any divisor contained in a rigid divisor is rigid itself, so there is an open dense set of $W_{h-1}(\Gamma)$ of such divisors. 
The space $W_{h-1}(\Gamma)$ is the Minkowski sum of $h-1$ copies of $W_1(\Gamma)$  in the Jacobian of $\Gamma$. Since the latter is the image of $\Gamma$ under the Abel--Jacobi map, the former is a connected polyhedral complex of pure dimension $h-1$.
Let $D$ be a non-rigid effective divisor of degree $h-1$. Then $D$ has a representative in which at least one chip is at a vertex. The set of classes of such divisors has dimension strictly smaller than $h-1$. Therefore, there is a dense open set in $W_{h-1}(\Gamma)$ classifying rigid divisors. 
Since the map taking a divisor $P$ to $K-P$ is a linear bijection between $W_{h-1}$ to itself, there is an open dense set for which both $P$ and $K-P$ are rigid. 
\end{proof}

The following lemma is a useful tool for dealing with the graph and algebraic parts of divisors sepearately  (cf. \cite[Theorem 4.3]{Len1}  and \cite[Proposition 2.1]{AB}).
\begin{lemma}\label{localRank}
Let $\calD$ be a divisor of rank $r>0$,  let $r=r_0 + \sum_{v\in V(\Gamma)} r_v$ be a partition of $r$, and let $E$ be an effective divisor of degree $r_0$ on $\Gamma$. Then  $\calD$ is equivalent to an effective divisor that contains $E$, whose restriction to each component of $\calC$ has rank $r_v$.
\end{lemma}
\begin{proof}
For each rational function $\calf$ on $\calC$ such that $\calD+\calf$  is effective and contains $E$, let $s_f$ be the collection of its incoming slopes at the components of $\calC$. Let $S$ be the set of all such $s_f$. %Note that there may be many different rational functions on $\calC$ with the same incoming slopes at all the nodes.
For $s\in S$ and a component $C_v$, let 
$$
T_{s,v} = \{D\in W_{r_v}(C_v)| \text{There exists $\calf$ with $s_\calf=s$ such that $D\leq \calD+\text{div}(\calf)$}\}.
$$ 
%That is, $T_{s,v}$ is the set of divisors of degree $g_v$ on $C_v$ that are contained in a divisor equivalent to $\calK$ whose slopes at $v$ are prescribed by $s$. 
The restriction to $C_v$ of divisors of the form $\calD+\text{div}(\calf)$ for $s_\calf=s$ are all equivalent to each other, and $T_{s,v}$  is the entire space $ W_{r_v}(C_v)$ if and only  those restrictions form a divisor class of rank $r_v$.

Let $T_s = \prod_{v\in V(\Gamma)} T_{s,v}$. Since $\calD$ has rank $r$, the union of all the sets $T_s$ is $ \prod_{v\in V(\Gamma)} W_{r_v}(C_v)$. Furthermore, since each $T_s$ is closed, and $S$ is finite, there must be $s\in S$ such that whenever $s_\calf=s$, the divisor $\calD+\text{div}(\calf)$ has rank $r_v$ at every component $C_v$. 
\end{proof}

For the rest of the section, we fix a metrized complex $\calC$ of genus $g>1$, whose underlying graph $\Gamma$ has genus $h$. Denote by $\calK$  the canonical class of $\calC$.

\begin{proposition}\label{canonicalRep}
Suppose that $h\geq 2$. Then there exist effective divisors $\calP, \calQ$ of degree $g-1$  such that:
\begin{enumerate}
\item  $\calP+\calQ\simeq \calK$.
\item  $\calP,\calQ$   are rigid.
\item For each component $C_v$ the restriction of $\calP,\calQ$ to $C_v$ is rigid and does not meet the nodes. 
\end{enumerate}
Moreover, there is an open set $\sigma$ classifying  such divisors in \linebreak$\prod_{v\in V(\Gamma)} W_{g_v}(C_v)\times W_{h-1}(\Gamma)$.
\end{proposition}

\begin{proof}
Fix rigid divisors $P+Q\simeq K$ on $\Gamma$ as in Lemma \ref{rigidDivisor}.  By Lemma \ref{localRank}, there is a representative $\calK'$ of $\calK$ that contains $P$, whose restriction to each component $C_v$ has rank $g_v$. 
Clifford's theorem for curves implies that this restriction has degree at least $2g_v$.
By subtracting $2g_v$ chips from each component, and forgetting the metrized complex structure, we  obtain a divisor which is equivalent to  the canonical divisor $K$ of $\Gamma$, and contains $P$. By Lemma \ref{rigidDivisor}, it is precisely the divisor $P+Q$. Therefore, the restriction of $\calK'$ to each component has exactly $2g_v$ chips. Since the set of rigid divisors of degree $g_v$ on $C_v$ is open and dense in $W_{g_v}(C_v)$, and the divisors $P,Q$ vary in  an open dense set of $W_{h-1}(\Gamma)$, there is an open dense set in the product satisfying the desired properties.
\end{proof}

\begin{proposition}\label{canonicalRepLow}
Suppose that $h$ is either $0$ or $1$. Then there exist effective divisors $\calP, \calQ$ of degree $g-1$  such that:
\begin{enumerate}
\item  $\calP+\calQ\simeq \calK$.
\item  $\calP,\calQ$   are rigid.
\item For each component $C_v$ the restriction of $\calP,\calQ$ to $C_v$ is rigid and does not meet the nodes. 
\end{enumerate}
Moreover, there is an  open set $\sigma$ classifying  such divisors in \linebreak$W_{g_0-1}(C_{v_0})\times\prod_{v_0\neq v\in V(\Gamma)} W_{g_v}(C_v)\times W_{h}(\Gamma)$ for some vertex $\text{ }v_0$.
\end{proposition}
\begin{proof}
Choose a vertex $v_0$ with $g_{v_0}>0$, which exists since we assumed that $g>1$. By Lemma \ref{localRank}, together with Clifford's theorem for curves, the divisor $\calK$ is equivalent to a divisor whose restriction to $C_{v_0}$ has degree $2g_{v_0}-2$ and rank $g_{v_0}-1$,  its restriction to any other component $C_v$ has degree $2g_v$ and rank $g_v$, and its restriction to the edges of the underlying graph $\Gamma$ has $2h$ additional chips. Moreover, the restriction to each component can be chosen so that it consists of a pair of rigid divisors. Now choose $\calP$ to consist of the first summands of those pairs and $h$ of the graph points, and $\calQ$ to consist of the second summands.   
\end{proof}

Recall that a set $\calR$ is said to be \emph{rank determining} if it suffices to consider points of $\calR$ when computing the ranks of divisors. More precisely, $\calR$ is rank determining if the rank of every divisor $\calD$ is  the largest number $r$ such that $\calD-p_1-\ldots-p_r$ is equivalent to an effective divisor, for every choice of $p_1,\ldots, p_r$ in $\calR$.
The following lemma is a mild generalization of Theorem A.1 in \cite{AB}. We leave it for the reader to make the necessary changes.
\begin{lemma}\label{RDS}
Let $\calR$ be a divisor of degree $g+1$ with the following properties:
\begin{enumerate}
\item $\calR$ has $h+1$ graph points (where $h$ is the genus of $\Gamma$), and the graph obtained from $\Gamma$ by removing $h$ of them is a tree.
\item The restriction of $\calR$ to every component $C_v$ is a rigid divisor of degree $g_v$.
\end{enumerate}
Then $\calR$ is a rank determining set.
\end{lemma}
\begin{flushright}
\qedsymbol
\end{flushright}

%\begin{proof}
%By \cite[Lemma A.3]{AB}, it suffices to show that the reduced divisor at every component has rank at least $1$. Fix a component $C_v$. By assumption, the reduced divisor $D_v$ may contain every point of a rigid divisor of degree $g_v$ on $C_v$. Then it either has rank 1, in which case we are done, or consists precisely of this divisor. Let us show that it cannot be the latter. Assume that it is. By \cite[Lemma A.4]{AB}, there is an open set $U$ on $\Gamma$ containing $v$ where $D_x$ is constant, and has degree 0, or rank 0 when $x$ is a vertex. $D_v$ has  rank $1$ at each of the boundary points of $U$. By the assumption on $\calR$,  every boundary point of $U$ has a chip if it's a graph point, and if it's a component either has rank $1$, or contains $g_u$. In the latter case it cannot be a boundary point, because the fire from nearby points goes freely through it. So we can basically ignore all the components, and have a potential RDS of degree $h+1$ that breaks the graph. But that is obvious.
%\end{proof}

\section{Hyperelliptic metrized complexes} \label{sec:proof}
\subsection{Clifford's theorem}
In this section, we assume the existence a divisor class $\delta$ of degree $2r$ and rank $r$ for some $1<r<g-1$, and conclude that $\calC$ is hyperelliptic. 
Let $\sigma$ be the open set  of rigid  divisors constructed in Proposition \ref{canonicalRep}, \ref{canonicalRepLow}. By construction, for each $\calP'\in\sigma$, there is a unique $\calQ'\in\sigma$ such that $\calP'+\calQ'\simeq\calK$. 
%Let $e_1,\ldots e_{h-1}$ be the edges containing the chips of $P$, and $f_1,\ldots,f_{h-1}$  the edges containing the chips of $Q$. Let $U_v$  the open set of $\text{Pic}_{g_v}(C_v)$ of divisors of rank $0$ that are supported away from the nodes and have no double points. 
%Denote  $\sigma=\prod_{i=1}^{h-1} e_i\prod_{v} U_v$  and $\tau=\prod_{i=1}^{h-1} f_i\prod_v U_v$. Each element in $\sigma,\tau$ represents a distinct effective divisor class of degree $g-1$. For each $\calP'\in\sigma$ there is a unique $\calQ'\in\tau$ such that $\calP'+\calQ'\simeq\calK$, as in lemma \ref{canonicalRep}. 
Let $\mu:\sigma\to\sigma$ be the map which assigns $\calQ'$ to $\calP'$. Fix a pair $\calP\in\sigma$ and $\calQ = \mu(\calP)$.

\begin{definition}
For divisors $\calD,\calE$, denote  
$$
(\calD\cap \calE)(v)=\min(\calD(v),\calE(v))
$$
and
$$
(\calD\cup \calE)(v)=\max(\calD(v),\calE(v)).
$$ 
For a divisor $\calD$, we define its $\calP$-\emph{part} as $\calD^\calP=\calD\cap \calP$, and its $\calQ$-\emph{part} as $\calD^\calQ=\calD\cap \calQ$.
\end{definition}

\begin{lemma}\label{uniqueness}
If $\calD\in\delta$ is effective and $\deg(\calD^\calP)=r$, then $\calD$ is supported on $\calP+\calQ$. In particular, $\calD=\calD^\calP+\calD^\calQ$.
\end{lemma}
\begin{proof}
By Riemann--Roch, the divisor $\calK-\calD$ has rank $g-1-r$. Let $\calE$ be an effective divisor that contains $\calP-\calD^\calP$ and is equivalent to $\calK-\calD$. Then $\calD+\calE$ is canonical, and as such, equivalent to $\calP+\calQ$. Therefore, $\calD+\calE - \calP \simeq \calQ$. But $\calD+\calE - \calP$ is effective and $\calQ$ is rigid, so $\calD+\calE-\calP$ is exactly $\calQ$. By adding $\calP$ to both sides, $\calD+\calE = \calP+\calQ$.
\end{proof}

By now, we know that there is a correspondence between the subsets of size $r$ of $\calP$ and $\calQ$. Next, we show that the correspondence respects unions and intersections.

\begin{proposition} \label{intersection}
If $\calD_1,\ldots, \calD_k$ are effective representatives of $\delta$ such that $\deg(\calD_i^\calP)=r$ for each $i$, then
$$
\deg(\calD_1^\calP\cap\ldots\cap\calD_k^\calP) = \deg(\calD_1^\calQ\cap\ldots\cap\calD_k^\calQ)
$$ and 
$$
\deg(\calD_1^\calP\cup\ldots\cup\calD_k^\calP) = \deg(\calD_1^\calQ\cup\ldots\cup\calD_k^\calQ).
$$
\end{proposition}
\begin{proof}
For each $i=1,\ldots,k$, let $\sigma_i$  be the subset of $\sigma$ consisting of divisors that contain $\calP-\calD_i^\calP$. Similarly, let $\tau_i$ be the subset of $\sigma$ of divisors containing $\calQ-\calD_i^\calQ$. 
%Then $\deg(\calD_i^\calQ\cap\ldots\cap\calD_k^\calQ) =\dim(\tau_1\cap\ldots\cap\tau_k)$, so to prove the first part of the proposition, it suffices to show that $\dim(\tau_1\cap\ldots\cap\tau_k) = \dim(\sigma_1\cap\ldots\cap\sigma_k)$.
We claim that the image of $\mu|_{\sigma_i}$ is contained in $\tau_i$ for each $i$.   Indeed, let $\calP_i\in\sigma_i$, and let 
$\calE_i = \calP_i - \calP + \calD_i^\calP$.
By definition of $\sigma_i$, the divisor $\calP_i$ contains $\calP-\calD_i^\calP$, so $\calE_i$ is effective (in fact, $\calE_i$ can be thought of as a divisor obtained by perturbing $\calD_i^\calP$). Since $\deg(\calE_i)=r$, there is a representative $\calD_i'$ of $\delta$ which contains $\calE_i$. Now, $\calP+\mu(\calP)+\calD_i' -\calD_i$ is effective because $\calD_i$ is contained in $\calP+\mu(\calP)$, it is equivalent to $\calK$ and contains $\calP_i$, so by the definition of $\mu$, it equals $\calP_i+\mu(\calP_i)$. It follows that $\mu(\calP_i)$ contains $\calK-\calD_i$, and in particular,  is in $\tau_i$. We conclude that the restriction of $\mu$ to $\sigma_1\cap\ldots\cap\sigma_k$ maps to $\tau_1\cap\ldots\cap\tau_k$. Since $\mu$ is a bijection, there is a one to one correspondence between divisors in $\sigma$ containing $\calP - \calD_1^\calP\cap\ldots\cap\calD_k^\calP$ and divisors containing $\calQ - \calD_1^\calQ\cap\ldots\cap\calD_k^\calQ$. In particular, $\calD_1^\calP\cap\ldots\cap\calD_k^\calP$ and $\calD_1^\calQ\cap\ldots\cap\calD_k^\calQ$ must have the same degree.

Now, by the  inclusion--exclusion principle, the cardinality of a union can be expressed as an alternating sum of the cardinalities of different intersection, which proves the second part of the proposition.
\end{proof}

We are finally in a position to prove our main theorem.

\begin{theorem}\label{mainTheorem}
Let $\calC$ be a metrized complex of genus $g$,
and suppose that for $0<r<g-1$ there is a divisor class $\delta$ of degree $2r$ and rank $r$. Then $\calC$ is hyperelliptic.
\end{theorem}
\begin{proof}
For each $p_i$ in $\calP$, Let
$$
S_i = \{\calD\in\delta | \deg(\calD^P) = r, p_i\notin \calD\}.
$$
Let $\phi$ be the map that assigns to every subset $A$ of size $r$ of $\calP$ the unique subset $B$ of $\calQ$ for which $A+B$ is in the divisor class $\delta$.
By Proposition \ref{intersection}, $\underset{\calD\in S_i}\cup\calD^\calQ$ contains all but a single point, $q_i$, of $\calQ$. Since the assignment $\calD^\calP\to \calD^\calQ$ is a bijection between the subsets of  size $r$ of $\calP$ and $\calQ$, we can reverse the process and conclude that $q_i\neq q_j$ when $i\neq j$. In particular, any divisor $\calD\in\delta$ whose  $\calP$-part has degree $r$, contains  $q_i$ if and only if it contains $p_i$. 
%Therefore, we obtain a bijection $\phi_0:\text{supp}(\calP)\to\text{supp}(\calQ)$ which maps $p_i$ to $q_i$, and such that $\phi(\{p_{i_1},\ldots,p_{i_r}\}) = (\{\phi_0(p_{i_1}),\ldots,\phi_0(p_{i_r})\})$.

We claim that all the divisors $p_i+q_i$ are equivalent.  Indeed, choose $\calD_i$ and $\calD_j$ such that $p_i+q_i\leq\calD_i$, $p_j+q_j\leq\calD_j$, and  $\deg(\calD^\calP_i\cap \calD^\calP_j)=r-1$. 
By the discussion above, $\deg(\calD_i\cap \calD_j)=2r-2$, so
%By Lemma \ref{intersection}, the rest of their $\calQ$-parts coincide. Therefore, 
$$0\simeq \calD_i - \calD_j = p_i+q_i - p_j-q_j,$$
so $p_j+q_j \simeq q_i+q_j$. To show that $p_1+q_1$ has rank $1$, we need to extend $\calP$ to a rank determining set. 

We first deal with the case $h\geq 2$.
Let $p_1,\ldots, p_{h-1}$ be the graph points of $\calP$. Find a point $p_h$ so that the complement of $p_1,\ldots, p_h$ in $\Gamma$ is a spanning tree. The point $p_h$ may be chosen so that $\calP'=\calP-p_1+p_h$ is  in $\sigma$. By construction, the assignments $\calD^\calP\to\calD^\calQ$ and $\calD^{\calP'}\to\calD^{\calQ'}$   coincide for any divisor $\calD\leq\calP\cap\calP'$ and therefore the bijection above extends to $\calP\cup\calP'$. In particular, we have an equivalence $p_1+q_1\simeq p_h+q_h$ for some $q_h$. 
Finally, let $p$ be any other graph point contained in a representative of $p_1+q_1$. By Lemma \ref{RDS},  $\calP\cup\{p_h,p\}$ is a rank determining set, so $p_1+q_1$ has rank $1$.

Now, suppose that $h=0,1$. Then by the construction in \ref{canonicalRepLow}, there is a vertex $v_0$ such that the restriction $P_0$ of $\calP$ to $C_{v_0}$ has degree $g_{v_0}-1$. Find a point $p_0$ so that $P_0+p_0$ is rigid, and a graph point $p$ that is contained in a representative of $p_1+q_1$. 
Similarly to the higher genus case,  $\calP\cup\{p_0,p\}$ is rank determining,  so $p_1+q_1$ has rank 1. 
The proof is complete.
\end{proof}

\subsection{The structure of hyperelliptic metrized complexes}
To complete our discussion on metrized complexes, we show that being hyperelliptic imposes strong conditions on their structure. The following characterization is familiar to experts, but to the extent of our knowledge, does not appear in the literature. For a visual illustration of the lemma,  see Figure \ref{example}.
\begin{lemma}\label{structureHE}
A metrized complex $\calC$ is hyperelliptic if and only if it satisfies the following properties:
\begin{enumerate}
\item The underlying graph $\Gamma$ is either a tree or hyperelliptic with involution $\iota_\Gamma$ (if $\Gamma$ is a tree then $\iota_\Gamma$ is just the identity).
\item If $C_v$ has genus $g_v>0$, then $\iota_\Gamma(v)=v$. For every node $p$  corresponding to an edge $e$, the edge $\iota_\Gamma(e)$  meets $C_v$ at a node $p'$, and all such divisors $p+p'$ are equivalent.
\item If $g_v\geq 2$, then $C_v$ is hyperelliptic with involution $\iota_v$, satisfying $\iota_v(p)=p'$ for every node $p$.
\end{enumerate}
\end{lemma}

\begin{proof}
Suppose that $\calC$ is hyperelliptic, and let $x+x'$ be a divisor of degree 2 and rank 1. 
When passing to the underlying graph,  rank may only increase, so $x+x'$ has rank at least 1 on $\Gamma$. By \cite[Theorem 1.3]{chan},  $\Gamma$ is either a tree or a hyperelliptic graph with involution $\iota_\Gamma$. Let $v$ be a vertex of $\Gamma$. By Lemma \ref{localRank},  there is a divisor equivalent to $x+x'$ whose restriction to  $C_v$ has rank at least 1. When $g_v>0$, it implies that this restriction has degree $2$. In particular, $v$ is a point of $\Gamma$ with $\iota_\Gamma(v)=v$.
When $g_v>1$, we conclude that $C_v$ is hyperelliptic with involution $\iota_v$. Let $p,p'$ be a pair of nodes on $C_v$ corresponding to edges $e'=\iota_\Gamma(e)$. Then for every pair of points $y,y'$ on $e,e'$ at equal distance $\epsilon$ from $v$, $x+x'\simeq y+y'$. By letting $\epsilon$ tend to zero, we see that $x+x'\simeq p+p'$.
 
Conversely, suppose that the conditions above are satisfied. Let $x$ be a graph point of $\Gamma$, and let $x'=\iota_\Gamma(x)$. For any other point $p$ of $\calC$, it can be verified, using Dhar's burning algorithm, that the $p$-reduced divisor equivalent to $p$ has a chip at $p$. In particular, the rank of $x+x'$ is $1$. \end{proof}

\noindent We define a map $\iota$ on $\calC$ as follows. If $p$ is a point of $\calC$ that does not lie on a rational component, then 
\[
 \iota(p) =
  \begin{cases} 
      \hfill \iota_\Gamma(p)    \hfill & \text{p is a graph point of $\Gamma$}. \\
      \hfill q \hfill & \hfill \text{$p$ is a point of $C_v$ with $g_v=1$, and $p+q$ is equivalent to a pair of nodes on $C_v$}. \\
      \hfill \iota_v(p) \hfill & \text{$p$ is a point of $C_v$ with $g_v\geq 2$}. \\
  \end{cases}
\]
If $p$ is on a rational component $C_v$, and $\iota_\Gamma(v)=v$, then define $\iota(p)=p$. Otherwise, $\iota(p)$ is any point of $C_{\iota(v)}$. As all the points on rational components are linearly equivalent, it does not matter, for purposes of divisor theory, which point we choose. 
For every $p\in\calC$, the rank of $p+\iota(p)$ is 1. Let $\calg^1_2$ be the divisor class of $p+\iota(p)$ for some $p$.

\subsubsection*{Reduced divisors} 
Recall that a divisor $\calD$ on a metrized complex is said to be $v$-\emph{reduced} with respect to a point $v$ of $\Gamma$ if it is effective away from $v$, and its chips are as "lexicographically close" as possible to $v$ (see \cite[Section 3.1]{AB} for a precise definition). For any $v$, reduced divisors exist and are quasi-unique, which means that their graph part is unique, and the restrictions to the components are unique up to linear equivalence \cite[Theorem 3.7]{AB}. 
We extend the definition to non-graphical points of a metrized complex. A divisor $\calD$ is $p$-\emph{reduced} for a point $p$ on a component $C_v$, if it is $v$-reduced,  and its restriction to $C_v$ has the highest degree at $p$ among the divisors that are effective away from $p$. Reduced divisors are quasi-unique, and their restriction to $C_v$ is unique.

%\cite{KY1}[Theorem 1.14]
In \cite{KY1}, it was shown that every divisor of degree $d$ and rank $r$ (for $0\leq r\leq d\leq g$) on a hyperelliptic metric graphs contains $r$ copies of the unique divisor of degree $2$ and rank $1$. A similar argument holds for metrized complexes.

\begin{theorem}\label{uniqueHyper}
Let $\calC$ be a hyperelliptic metrized complex. Then every divisor class $\delta$ of degree $d$ and rank $r$ with $0\leq r\leq d\leq g$ is equivalent to $r\cdot\calg^1_2 + p_{2r+1}+\ldots+p_d$. 
\end{theorem}
\begin{proof}
Let $p$ be a point of $\calC$ with $\iota(p)=p$. Let $\calD$ be the $p$-reduced representative of $\delta$. It suffices to show that $\calD(p)\geq 2r$.
First, assume that $r=1$. In this case, $\calD(p)\geq 1$. Assume for the sake of contradiction that it equals $1$. Since $d\leq g$, Lemma 3.11 of \cite{AB} implies that there is a point $q\neq p$ of $\calC$ such that  $\calD+q$ is still $p$-reduced. 
Let $\calD'$ be the $p$-reduced representative of  $\delta+p-\iota(q)$, and set $\calD''=\calD+q-p$. Then $\calD''\simeq\calD'$, and both are $p$-reduced, because $p$-reducedness does not change when adding or subtracting chips at $p$. By  quasi-uniqueness, both divisors have the same number of chips at $p$. But  $\calD''(p)=0$ and $\calD'(p)$ is at least 1, so we arrive at a contradiction.

For $r>1$, let $s$ be the largest integer so that $\calD = 2s\cdot p + q_1+\ldots+q_k$ (where one of the points $q_i$ might coincide with $p$). We need to show that $s\geq r$. Again, by  \cite[Lemma 3.11]{AB}, there are points $q_{k+1},\ldots,q_{k+s}$ different from $p$, such that $\calD+q_{k+1}+\ldots+q_{k+s}$ is $p$-reduced. Since $2p$ is equivalent to $q+\iota(q)$ for every $q$, we have
 $\calD\simeq q_1+\ldots+q_k+q_{k+1}+\iota(q_{k+1})+\ldots + q_{k+s} + \iota(q_{k+s})$. 
Therefore, the $p$-reduced representative of $\calD - \iota(q_{k+1}) -\ldots - \iota(q_{k+s})$ is exactly $q_1+\ldots+q_k$.
Its degree at $p$ is at most  $1$, so by the first part, its rank is 0. This divisor was obtained from $\calD$ by removing $s$ chips, so the rank of $\calD$ is at most $s$.

\end{proof}

\appendix
\begin{section}{Martens' theorem (joint with David Jensen)}
In this section, we discuss  possibilities for a tropical version of Martens' theorem, which refines the characterization of hyperelliptic curves provided by Clifford. It is one of several structure theorems for classifying special curves according to their Brill--Noether loci. Further refinement is given by Mumford \cite[Theorem 5.2]{ACGH}, and on the other extreme, the Brill--Noether theorem determines the dimension of the Brill--Noether locus of general curves \cite{Kempf, KL, GH}.
Let us first recall the classical statement of Martens' theorem.

\begin{theorem*}\cite[Theorem 5.1]{ACGH}
Let $C$ be a smooth curve of genus $g$, and let $d,r$ be integers satisfying $0<2r\leq d<g$.
Then $\dim(W^r_d(C))\leq d-2r$, and equality holds precisely when $C$ is hyperelliptic.  
\end{theorem*}
\noindent Note that the special case where $d=2r$ is  Clifford's theorem. 
As a first attempt, we examine a naive tropical analog of the theorem. As the following result shows, the first part of the statement holds, but the second, unfortunately, does not. 
\begin{theorem}%[joint with Dave Jensen]
Let $\Gamma$ be a metric graph of genus $g$, and let $d,r$ be integers satisfying $0<2r\leq d<g$.
The dimension of $W^r_d(\Gamma)$ is at most $d-2r$, and equality holds for hyperelliptic graphs. However, the dimension of  $W^r_d(\Gamma)$ may be $d-2r$ for non-hyperelliptic curves.
\end{theorem}
\begin{proof}
We prove the first part by induction on  $r$. Let $r=1$. We need to show that $\dim(W^1_d(\Gamma))< d-1$. Since $r=1$,  the Brill--Noether locus is the set of divisor classes with an effective representative through every point of $\Gamma$. That is, 
$$
W^1_d(\Gamma) = \underset{p\in\Gamma}{\cap}{p + W_{d-1}(\Gamma)}.
$$
\noindent  As we mention in the proof of Lemma \ref{rigidDivisor}, the space $W_{d-1}(\Gamma)$ is a the Minkowski sum of $d-1$ copies of the Abel--Jacobi image of the $\Gamma$, and is therefore
a connected polyhedral complex of dimension $d-1$. Moreover, since the sum of $g$ copies of the Abel--Jacobi image is the entire Jacobian, $W_{d-1}(\Gamma)$ cannot be contained in a subtorus. Therefore, for every point $\delta$ of $W^1_d(\Gamma)$, we can find a translation $p+W_{d-1}(\Gamma)$ of $W_{d-1}(\Gamma)$ that intersects $W_{d-1}(\Gamma)$ in a dimension strictly smaller than $d-1$ in a neighborhood of $\delta$. 

Next, assume that the claim holds for $r-1$.
We have 
$$
W^r_d(\Gamma) =  \underset{p\in\Gamma}{\cap}{p + W^{r-1}_{d-1}(\Gamma)}.
$$
By induction, all the cells of $W^{r-1}_{d-1}(\Gamma)$ are of dimension at most $d-2r+1$. Again, find $p$ such that the the intersection of $W^{r-1}_{d-1}(\Gamma)$ and  $p+W^{r-1}_{d-1}(\Gamma)$  around $\delta$ is not full dimensional. 

Now, suppose that $\Gamma$ is hyperelliptic, and let $\delta$ be an element of $W^r_d(\Gamma)$. By Proposition \ref{uniqueHyper}, $\delta$ has a representative of the form $w_1+\ldots + w_{d-2r} + 2r\cdot v_0$, where $v_0$ is a point of $\Gamma$ which is invariant under the hyperelliptic involution. Since $2r\cdot v_0$ has rank $r$, any perturbation of the points $w_1,\ldots,w_{d-2r}$ results in a divisor of rank at least $r$. It follows that $\delta$ has a neighborhood of dimension $d-2r$ in $W^r_d(\Gamma)$.

To show that the converse is false, let $\Gamma$ be the graph from Theorem 1.2 in \cite{LPP}. It is straightforward to check that $\Gamma$ is not hyperelliptic. However, the dimension of $W^1_3(\Gamma)$ is $1=3-2\cdot 1$.
\end{proof}

The theorem above is not the first example in which the dimension of the Brill--Noether locus exhibits unpleasing behavior. For instance, it does not vary upper semicontinuously on the moduli space of tropical curves  \cite[Theorem 1.2]{LPP}. These phenomena suggest that a different quantity should act as the tropical analog for the dimension of the Brill--Noether locus. Such an invariant was introduced in \cite{LPP}.
\begin{definition}
The Brill--Noether rank, denoted $w^r_d$ is the largest number $\rho$ such that every effective divisor of degree $r+\rho$ is contained in a divisor of degree $d$ and rank $r$.
\end{definition}

For an algebraic curve, the Brill--Noether rank coincides with the dimension of the largest component of its Brill--Noether locus. Consequently, it satisfies a specialization lemma: if $\Gamma$ is the skeleton of an algebraic curve $C$, then $w^r_d(\Gamma)\geq \dim(W^r_d(C))$ . Furthermore, it varies upper semicontinuously on the moduli space of tropical curves (Theorems 5.1, 5.3 of \cite{Len1}, and Theorems 1.6, 1.7 of \cite{LPP}).   
As we show here, it also satisfies the first part of Martens' theorem.

\begin{proposition}
Let $\Gamma$ be a metric graph of genus $g$, and let $r,d$ be as in the conditions of Martens' theorem. Then 
$$
w^r_d(\Gamma)\leq d-2r.
$$
Moreover, for hyperelliptic graphs, $w^r_{d}(\Gamma)=d-2r$.
\end{proposition}
\begin{proof}
Suppose for contradiction that $w^r_d > d-2r$. Then every divisor of degree $d-r+1$ is  contained in a divisor of degree $d$ and rank $r$. But this is clearly false: choose a divisor of degree $d-r+1$ and rank 0. By adding $r-1$ points to any divisor, the rank can  increase by at most $r-1$. 

For the second part, assume that $\Gamma$ is hyperelliptic, and let $E$ be an effective divisor of degree $d-r$. We need to show that it is contained in a divisor of degree $d$ and rank at least  $r$.
Since $d\geq 2r$, the degree of $E$ is at least $r$. Let $p_1,\ldots, p_r$ be points in the support of $E$, and let $\iota$ be the hyperelliptic involution of $E$. Then $E+ \iota(p_1)+\ldots +\iota(p_r)$ has rank at least $r$. 
\end{proof}

Given the facts above, we speculate that Martens' theorem holds in tropical geometry. 
\begin{conjecture*}
Let $\Gamma$ be a metric graph of genus $g$, and let $d,r$ be such that $0<2r\leq d<g$.
Then $w^r_d(\Gamma)\leq d-2r$, and equality holds precisely exactly when $C$ is hyperelliptic.  
\end{conjecture*}
\end{section}

\bibliographystyle{plain}
\bibliography{Cliffbib}

\end{document}